\documentclass[12pt]{article}   	
\usepackage{geometry}                		
\usepackage{graphicx}				
\usepackage{caption}
\usepackage{subcaption}
\usepackage{amssymb}
\usepackage{amsthm}
\usepackage{csquotes}
\usepackage{amsmath}
\usepackage{mathtools}
\usepackage{float}
\usepackage{physics}
\MakeOuterQuote{"}
\newtheorem*{theorem*}{Theorem}
\newtheorem{definition}{Definition}
\newtheorem{theorem}{Theorem}[section]

\newtheorem{lemma}[theorem]{Lemma}

\newtheorem*{conjecture*}{Conjecture}
\title{An Application of the Theory of Viscosity Solutions to Higher Order Differential Equations}
\author{Matei P. Coiculescu \thanks{Supported by an N.S.F. Graduate Research Fellowship}}
\begin{document}
\maketitle
\begin{abstract}
We directly apply the theory of viscosity solutions to partial differential equations of order greater than two. 
We prove that there exists a solution in $C^{2,\alpha}(B_R)\cap C(\overline{B_R})$ for the inhomogeneous $\infty$-Bilaplacian equation on a ball $B_R\subset \mathbb{R}^n$:
$$\Delta_\infty^2 u:=(\Delta u)^3 \abs{D(\Delta u)}^2 =f(x)$$
with Navier Boundary conditions ($u=g\in C(\partial B_R), \Delta u =0 \textrm{ on } \partial B_R$). We also prove that there exists a solution in $C^{1,\alpha}(\mathbb{R}^n)$ for all $\alpha>0$ to the eigenvalue problem on $\mathbb{R}^n$:
$$\Delta_\infty^2 u =-\lambda u+f(x)$$
whenever $n\geq 3, \lambda<0,$ and $f(x)$ is continuous, bounded, and supported on an annulus. 
\end{abstract}
\section{Introduction}

The development of the theory of viscosity solutions for solving fully nonlinear partial differential equations has lead to strikingly general existence, uniqueness, and regularity results for second order problems. On the other hand, since the definition of a viscosity solution depends on what is essentially a second order phenomenon (i.e. some version of the maximum principle), the direct application of this theory has always been limited to second order differential equations. In this paper, we examine compositions of differential operators of order less than or equal to two and directly apply the theory of viscosity solutions to these higher order differential equations that are in this "factored" form.

Generally, we have differential operators $L_0, L_1,\ldots, L_k$ of order less than or equal to two, and we consider the equation
$$(L_k \circ \cdots \circ L_1\circ L_0) u=0.$$
Commonly, when $L_k$ consists only of pure derivatives, an equation of this form is said to be in divergence form. Thus, the differential equations we deal with are in a form related to, yet more general than equations in divergence form.

We must also take care in applying the theory of viscosity solutions to such a composition of operators, since viscosity solutions encode only "indirect" information and local behavior of the respective equation. Indeed, we can observe that the use of the theory of viscosity solutions for solving $L_j$ requires the use of viscosity solutions for all $L_{j+i}$ with  $i\geq 0$ as well.

We provide an explicit application of the theory of viscosity solutions to a higher order differential equation by solving the inhomogeneous $\infty$-Bilaplacian equation:
$$\Delta_\infty^2 u:=(\Delta u)^3 \abs{D(\Delta u)}^2 =f(x).$$
This equation is the formal limit as $p\to \infty$ of the $p$-Bilaplacian equation
$$\Delta_p^2 u = \Delta(\abs{\Delta u}^{p-2}\Delta u)= f(x),$$
which arises from the minimization problem for the $L^p$ norm of $\Delta u$. In other words, the $p$-Bilaplacian is the Euler-Lagrange equation for the functional
$$\int_\Omega \abs{\Delta u}^p dx.$$
Minimizers are sought in the space $W^{2,p}(\Omega)$, and the weak formulation of the homogeneous Dirichlet Problem for the $p$-Bilaplacian is solvable in $W^{2,p}$ whenever $p>2$. Likewise, the $\infty$-Bilaplacian is the Euler-Lagrange equation for the $L^\infty$ variational problem of minimizing $\|\Delta u\|_\infty$, and minimizers are sought in the space $W^{2,\infty}$. The existence of minimizers for the $p$-Bilaplacian (weak solutions of $\Delta_p^2$) with $2\leq p<\infty$ follows relatively easily from the coercivity of the related semilinear form:
$$(u,v) \to \int_\Omega (\abs{\Delta u}^{p-2}\Delta u )\Delta v.$$
Further information on the $p$-Bilaplacian may be found in \cite{KP}.
The case $p=2$ is the linear Bilaplacian equation $\Delta^2 u=0$. Among the usual choices for boundary values (Dirichlet, Neumann, Mixed), the Bilaplacian admits existence and uniqueness results for its weak formulation with
\textbf{Navier Boundary Conditions}:
$$u=g, \quad \Delta u= 0 \textrm{ on } \partial \Omega$$
Indeed, with these boundary conditions, the Bilaplacian $\Delta^2$ can be appropriately decomposed into two iterations of Poisson's Equation, and the natural function space to search for solutions becomes $W^{2,2}$. A discussion about the Bilaplacian along these lines is presented in \cite{S}. Navier Boundary Conditions are the best choice for the equations we study, since we would also like to decompose higher order problems into second order ones. In particular, we shall analyze the Navier Boundary Problem for the $\infty$-Bilaplacian and perturbations of the same, and we obtain the following theorems:
\begin{theorem}
Let $\epsilon>0$ be arbitrary. Let $B_R$ be the ball of radius $R$ around the origin of $\mathbb{R}^n$. Let $f\geq 0$ be in $C^{0,\alpha}(B_R)\cap C(\overline{B_R})$ and $g$ be a continuous function on $\overline{B_R}$. We assume in addition that
$$f(x)= O\big((R^2-\abs{x}^2)^3\abs{x}^2\big).$$
Then the boundary value problem
$$
\begin{cases}
\epsilon \Delta v+(\Delta v)^3 \abs{D(\Delta v)}^2=f(x)\\
v=g \textrm{ on } \partial B_R\\
\Delta v = 0 \textrm{ on } \partial B_R\\
\end{cases}
$$
has a viscosity solution $v_\epsilon$ in $C^{2,\beta}(B_R)\cap C(\overline{B_R})$, where $0<\beta \leq \alpha$.
\end{theorem}
The solutions to the perturbed problem will help us construct a solution to the unperturbed equation:
\begin{theorem} Let $f,g$ satisfy the hypotheses of the previous theorem.
The boundary value problem
$$
\begin{cases}
(\Delta v)^3 \abs{D(\Delta v)}^2 = f(x)\\
v=g \textrm{ on } \partial B_R\\
\Delta v = 0 \textrm{ on } \partial B_R\\
\end{cases}
$$
has a viscosity solution $v \in C^{2,\beta}(B_R)\cap C(\overline{B_R})$, where $0<\beta \leq \alpha$.
\end{theorem}

The $\infty$-Bilaplacian equation does not admit classical $C^3$ solutions for the Dirichlet Problem unless we are in the homogeneous case with very simple boundary values (e.g. quadratic functions when in $\mathbb{R}$) \cite{KP2}. 
On the other hand, our theorem shows that the $\infty$-Bilaplacian equation admits a viscosity solution to the Navier Problem with $C^{2,\alpha}$ regularity, which is as regular as we might expect. The solution we find is also in the space $W^{2,\infty}$, which is the space desired for the associated variational problem. In this way, we improve on the existing theory detailed, for example, in \cite{KP}, \cite{KP2} by dealing with the inhomogeneous case and by the generality of our boundary values for $v$. For example, the authors of \cite{KP2} have shown existence for the Dirichlet Problem with boundary values in $W^{2,\infty}$, while we deal with continuous boundary values. We also remark that the homogeneous Navier Problem for the $\infty$-Bilaplacian admits very simple solutions. Indeed, we have classical solutions that coincide in this case with harmonic functions with appropriate boundary values. In particular:

\begin{theorem}
Let $\Omega$ be a $C^2$ bounded domain, and let $g$ be a continuous function on $\overline{\Omega}$. Then the Navier boundary value problem:
$$
\begin{cases}
(\Delta v)^3 \abs{D(\Delta v)}^2 = 0\\
v=g \textrm{ on } \partial \Omega\\
\Delta v = 0 \textrm{ on } \partial \Omega\\
\end{cases}
$$
admits a unique classical solution $v$ that is precisely the solution to the Dirichlet problem $\Delta v=0$ in $\Omega$, $v=g$ on $\partial \Omega$. 
\end{theorem}

Lastly, we also prove an existence theorem for an eigenvalue problem associated to the $\infty$-Bilaplacian:
\begin{theorem}
Let $n\geq 3$. For any $\lambda<0$ and for any continuous and bounded function $f$ with support on some open annulus centered at the origin, there is a viscosity solution $u$ in  $C^{1,\alpha}(\mathbb{R}^n)$ for all $\alpha>0$ to:
$$\lambda u +(\Delta u)^3\abs{D(\Delta u)}^2=f(x).$$
\end{theorem}
More generally, the introduction of a lower order term requires us to solve an integro-differential equation using the theory of viscosity solutions.

The generality of existence theorems in the theory of viscosity solutions allows us to obtain results that apply to a much wider class of differential equations, but we omit this discussion here. Instead, we focus on concrete examples with emphasis on the $\infty$-Bilaplacian equation.

In Section 2, we introduce the use of viscosity solutions for a simple higher order example: a nonlinear third-order ordinary differential equation. In Section 3, we analyze the $\infty$-Bilaplacian equation, and in Section 4 we discuss the addition of a zeroth-order term. Our last section includes an application to an eigenvalue problem associated to the the $\infty$-Bilaplacian.

We would like to thank our advisor Camillo De Lellis for very helpful conversations about the theory of viscosity solutions and about this paper. We also thank Ravi Shankar for his insightful comments and his encouragement to improve and publish this work.
\section{A First Example}
Before discussing our results for partial differential equations, we shall find it instructive to examine a simpler model: a fully nonlinear third-order ordinary differential equations. More precisely, we consider the following boundary value problem:\\
\begin{equation}
F(x, u_{xxx}) =0 \textrm{ on } [-\pi, \pi]
\end{equation}
\begin{equation}
u(-\pi)=u(\pi), \quad u_x(-\pi)=u_x(\pi)
\end{equation}

We now make the (formal) differential transformation $u_x=w$ to work with a second-order differential equation instead. More precisely, if $w(x)$ is a classical $C^2$ solution of 
\begin{equation}
F(x, w_{xx}) =0 \textrm{ on } [-\pi, \pi]
\end{equation}
\begin{equation}
w(-\pi) = w(\pi)
\end{equation}
we define 
\begin{equation}
u(x) := u(-\pi) +\int_{-\pi}^x w(s) ds 
\end{equation}
which shall be a classical $C^3$ solution of $(1), (2)$ if, in addition to solving $(3), (4)$, $w$ also has zero mean:
\begin{equation}
\int_{-\pi}^{\pi} w(s) ds =0.
\end{equation}
The last condition $(6)$ is satisfied, for example, if the function $w(x)$ is odd.

We have the following definition:
\begin{definition} The function $u(x)$ is a $C^1$ {\textbf{viscosity}} solution to
$$F(x, u_{xxx}) =0 \textrm{ on } (-\pi, \pi), \quad u(-\pi)=u(\pi), \quad u_x(-\pi)=u_x(\pi)$$
if and only if
$$u(x) = u(-\pi) +\int_{-\pi}^x w(s) ds$$
where $w(x)$ is a continuous viscosity solution to  
$$F(x, w_{xx}) =0 \textrm{ on } (-\pi, \pi), \quad w(-\pi)=w(\pi).$$
\end{definition}

In what follows, we determine sufficient conditions for the functional $F$ so that the boundary value problem $(1), (2)$ has a unique $C^1$ viscosity solution. Namely:
\begin{theorem}
Let $F(x,s):[-\pi,\pi]\times\mathbb{R}\to \mathbb{R}$ be a continuous function that is strictly monotonic in the second variable, in the sense that there exists $\lambda>0$ so that:
$$F(x,s+t)\geq F(x,s) +\lambda t.$$
 Assume in addition that we have the following symmetry condition:
$$F(-x,-s)=-F(x,s), \textrm{ for all permissable } (x,s).$$
It follows that the boundary value problem
$$-F(x, u_{xxx}) =0, \quad u(-\pi)=u(\pi), \quad u_x(-\pi)=u_x(\pi)$$
has a unique $C^1$ viscosity solution. 
\end{theorem}
\begin{proof}
By the assumptions on the functional $F(x,s)$, we know that the problem:
$$-F(x, w_{xx}) =0, \quad w(-\pi) = w(\pi)$$
has a unique continuous viscosity solution (cf. Theorem 5.25 in \cite{ACM}). We now show this solution must also be an odd function on $[-\pi, \pi]$. Indeed, let $\hat{w}(x)=-w(-x)$. Our goal is to show that $\hat{w} = w$ by showing that $\hat{w}$ is also a continuous viscosity solution. 

Suppose that $\hat{w}-\phi$ has a local maximum at $x$, which implies that $-w(-x)-\phi(x)$ has a local maximum at $x$. Then $w(-x)+\phi(x)$ has a local minimum at $x$. Let $\hat{\phi}(x)=-\phi(-x)$. Then we see that $w(-x)-\hat{\phi}(-x)$ has a local minimum at $x$, or equivalently, $w(x)-\hat{\phi}(x)$ has a local minimum at $-x$. It follows from the supersolution criterion that $-F(-x, \hat{\phi}_{xx})\geq 0$, which implies $-F(-x,-\phi_{xx})=F(x,\phi_{xx})\geq 0$, which in turn gets $-F(x,\phi_{xx})\leq 0$. Since $\phi$ was arbitrary, $\hat{w}$ is a subsolution of the same equation. By the same argument, since $w$ is a subsolution, $\hat{w}$ is a supersolution. Lastly, by our assumptions on the functional $F$, continuous viscosity solutions of our equation are unique, which means that $\hat{w}=w$, i.e. $w$ is odd.

Since $w$ is odd, we have that
$$\int_{-\pi}^{\pi} w(s)ds=0$$
and the function
$$u(x) = u(-\pi) +\int_{-\pi}^x w(s)ds$$ is our unique $C^1$ viscosity solution to the boundary value problem.
\end{proof}
As noticed before, we really only need that $w$ has zero mean, so the hypotheses of the theorem could probably be weakened while retaining an existence result for $C^1$ viscosity solutions to the problem.
We finish the section with two typical theorems relating viscosity and classical solutions.
\begin{theorem}
Let $F(x,s)$ be as in the statement of Theorem 2.1. If the boundary value problem 
$$F(x, u_{xxx}) =0, \quad u(-\pi)=u(\pi), \quad u_x(-\pi)=u_x(\pi)$$
has a classical $C^3$ solution, then it is the unique $C^1$ viscosity solution.
\end{theorem}
\begin{proof}
Let $u$ be a $C^3$ classical solution of the problem. Then, $w(x) =u_x(x)$ is a $C^2$ function solving 
$$-F(x, w_{xx}) =0$$
in the classical sense. Since $F(x,s)$ is continuous and monotonically increasing in $s$, it follows from Theorem 5.11 in \cite{ACM} that $w$ is a viscosity solution. This in turn implies that $u$ is a viscosity solution since, by construction, we may write $u$ as
$$u(x) =u(-\pi) +\int_{-\pi}^x w(s) ds.$$
\end{proof}
\begin{theorem}
If a viscosity solution of
$$F(x, u_{xxx}) =0, \quad u(-\pi)=u(\pi), \quad u_x(-\pi)=u_x(\pi)$$
 is in $C^3$, then it is a classical solution.
\end{theorem}
\begin{proof}
Let $u$ be a viscosity solution. Then there is a continuous function $w$ that is a viscosity solution of 
$$-F(x, w_{xx})=0$$
such that
$$u(x) =u(-\pi) +\int_{-\pi}^x w(s) ds.$$
Since $u$ is in $C^3$, it follows that $w$ is in $C^2$. Thus, by Theorem 5.10 in \cite{ACM}, $w$ is also a classical solution of the problem, which immediately implies the same for $u$.
\end{proof}
\section{The $\infty$-Bilaplacian Equation}
We study the Navier problem for the inhomogeneous $\infty$-Bilaplacian equation:
\begin{equation}
\Delta_\infty^2 u = (\Delta u)^3 \abs{D(\Delta u)}^2=f(x)
\end{equation}
As discussed earlier, the $\infty$-Bilaplacian is the formal limit of the $p$-Bilaplacian equation as $p\to \infty$. The latter equation is a nonlinear generalization of the Bilaplacian (Biharmonic) equation $\Delta^2 u=0$. Both equations have many applications, and we refer the reader to \cite{LM} and \cite{S} for a more complete discussion of their use. Also, the $\infty$-Bilaplacian is the Euler-Lagrange equation for minimizing $\| \Delta u\|_\infty$, which is the prototypical $L^\infty$ variational problem and has been studied extensively in \cite{KP} and \cite{KP2}, among others.

Typically, a Dirichlet problem for $\Delta_\infty^2$ admits neither a classical $C^3$ solution nor a weak $C^2$ solution, even for domains in $\mathbb{R}$. However, in what follows we show that a $C^{2,\alpha}$ solution to the Navier Problem may be found that is constructed from (and bounded by) a sequence of $C^{2,\alpha}$ viscosity solutions to a perturbation of $\Delta_\infty^2$.

\begin{theorem}
Let $\epsilon>0$ be arbitrary. Let $B_R$ be the ball of radius $R$ around the origin of $\mathbb{R}^n$. Let $f\geq 0$ be in $C^{0,\alpha}(B_R)\cap C(\overline{B_R})$ and $g$ be a continuous function on $\overline{B_R}$. We assume in addition that
$$f(x)= O\big((R^2-\abs{x}^2)^3\abs{x}^2\big).$$
Then the boundary value problem
$$
\begin{cases}
\epsilon \Delta v+(\Delta v)^3 \abs{D(\Delta v)}^2=f(x)\\
v=g \textrm{ on } \partial B_R\\
\Delta v = 0 \textrm{ on } \partial B_R\\
\end{cases}
$$
has a viscosity solution $v_\epsilon$ in $C^{2,\beta}(B_R)\cap C(\overline{B_R})$ where $0<\beta \leq \alpha$.
\end{theorem}
\begin{proof}
The boundary value problem
$$
\begin{cases}
\epsilon w+w^3 \abs{Dw}^2=f(x)\\
w= 0 \textrm{ on } \partial B_R\\
\end{cases}
$$
can be seen to admit a continuous viscosity solution by Perron's Method. Indeed since $f\geq 0$, $w_0=0$ is an obvious subsolution to the problem. it remains to construct a supersolution $w_1$ such that $w_1 =0$ on $\partial B_R$. By assumption, there exists a constant $C$ such that
$$0\leq f(x) \leq C\big((R^2-\abs{x}^2)^3\abs{x}^2\big)$$
for all $x\in \overline{B_R}$. Let $M$ be such that:
$$4M^5 >C.$$
The function $w_1(x) = M(R^2-\abs{x}^2)$ is now our desired supersolution. Obviously, $w_1=0$ on the boundary of the sphere, $\partial B_R$. It remains to show that it satisfies the supersolution condition. Let $\varphi \in C^\infty_0$ be arbitrary such that $w_1 - \varphi$ has a local minimum at $x\in \Omega$. Without loss of generality, the value of the local minimum is zero, and because $w_1$ is smooth as well, we necessarily have $Dw_1(x)=D\varphi(x)$. Thus, it suffices to show that
$$w_1^3(x)\abs{Dw_1(x)}^2\geq C\big((R^2-\abs{x}^2)^3\abs{x}^2\big)\geq  f(x)$$
since $w_1\geq 0$ and
$$\epsilon w_1+w_1^3(x)\abs{Dw_1(x)}^2 \geq w_1^3(x)\abs{Dw_1(x)}^2.$$
To that end, we calculate:
$$\abs{Dw_1(x)}^2 = \abs{2M x}^2 = 4M^2 \abs{x}^2$$
and
$$w_1^3(x) = M^3(R^2-\abs{x}^2)^3.$$
Thus,
$$w_1^3(x)\abs{Dw_1(x)}^2 \geq 4M^5\abs{x}^2(R^2-\abs{x}^2)^3\geq C\big((R^2-\abs{x}^2)^3\abs{x}^2\big)\geq  f(x).$$
Now Perron's method for viscosity solutions guarantees the existence of a continuous solution $w(x)$. The H{\"o}lder regularity and uniqueness follows from Theorem 2.6 in \cite{JK}. Applying this theorem requires a monotonicity condition satisfied because we have introduced the perturbation term. Additionally, we have that in $\overline{B_R}$: 
$$0\leq w(x) \leq M(R^2-\abs{x}^2).$$
Then, since $w_\epsilon$ is H{\"o}lder continuous (say of exponent $0<\beta\leq \alpha$), we can solve the Dirichlet problem:
$$
\begin{cases}
\Delta v =w_\epsilon\\
v=g \textrm{ on } \partial B_R
\end{cases}
$$
and get a viscosity solution $v_\epsilon\in C^{2,\beta}(B_R)\cap C(\overline{B_R})$ to the problem.
\end{proof}
\begin{theorem} Let $f,g$ satisfy the hypotheses of Theorem 3.1. Then the boundary value problem
$$
\begin{cases}
(\Delta v)^3 \abs{D(\Delta v)}^2 = f(x)\\
v=g \textrm{ on } \partial B_R)\\
\Delta v = 0 \textrm{ on } \partial B_R\\
\end{cases}
$$
admits a viscosity solution $v\in C^{2,\beta}(B_R)\cap C(\overline{B_R})$, where $0<\beta \leq \alpha$.
\end{theorem}
\begin{proof}
In the proof of the previous theorem, the viscosity solutions $w_\epsilon$ that solve
$$\epsilon w +w^3 \abs{Dw}^2= f(x)$$
have $\beta$-H{\"o}lder norm that depends only on the $\alpha$-H{\"o}lder norm of $f(x)$. By the Arzela-Ascoli theorem, there is a subsequence of $w_\epsilon$ converging uniformly to a function $w\in C^{0,\beta}(B_R)$ that is a viscosity solution of the unperturbed equation
$$w^3\abs{Dw}^2=f(x).$$
Finally, solving the corresponding Poisson equation yields our desired solution $v\in C^{2,\alpha}(B_R)\cap C(\overline{B_R})$.
\end{proof}
Finally, we deal with the homogeneous Navier boundary value problem:
\begin{theorem}
Let $\Omega$ be a $C^2$ bounded domain, and let $g$ be a continuous function on $\overline{\Omega}$. Then the Navier boundary value problem:
$$
\begin{cases}
(\Delta v)^3 \abs{D(\Delta v)}^2 = 0\\
v=g \textrm{ on } \partial \Omega\\
\Delta v = 0 \textrm{ on } \partial \Omega\\
\end{cases}
$$
admits a unique classical solution $v$ that is precisely the solution to the Dirichlet problem $\Delta v=0$ in $\Omega$, $v=g$ on $\partial \Omega$. 
\end{theorem}
\begin{proof}
The problem
$$
\begin{cases}
w^3 \abs{Dw}^2=0\\
w=0 \textrm{ on } \partial \Omega\\
\end{cases}
$$
is easily seen to admit $w=0$ as its unique solution. Finding a solution to the original problem is now tantamount to solving the standard Dirichlet problem $\Delta v =0$ in $\Omega$, $v=g$ on $\partial \Omega$.  The solution to this is the harmonic function $v$ with boundary values $g$, and $v$ is in particular an absolute minimizer of $\| \Delta u\|_\infty$, which can be forthrightly seen.
\end{proof}
Theorems 3.1 and 3.2 can probably be extended to sufficiently smooth bounded and star-shaped domains while maintaining the same supersolution construction; however, the inhomogeneity $f$ should have a similar behavior at the the boundary of the domain.
\section{General Equations with a Zeroth-Order Term}
In this section we deal with the whole space $\mathbb{R}^n$. We will prove a general Perron-type existence theorem for the following general equation:
\begin{equation}
F(x,u,\Delta u, D(\Delta u), D^2(\Delta u))=0
\end{equation}
which we will consider as a system of coupled equations:
\begin{equation}
\label{ZeroEqn}
\begin{gathered}
F(x, \Delta^{-1}w, w, Dw, D^2w)=0.\\
\Delta u = w
\end{gathered}
\end{equation}
The functional $F$ will always be assumed continuous, degenerate elliptic:
$$F(\cdot,\cdot,\cdot,\cdot, X)\geq F(\cdot,\cdot,\cdot,\cdot, Y) \textrm{ whenever } X-Y \textrm{ is negative definite,}$$
and nonincreasing in the $\Delta^{-1}w$ variable. 
We have denoted by $\Delta^{-1}$ the operator of convolution with $\Psi_n$:
$$\Delta^{-1} w := w\ast \Psi_n$$
where 
\begin{equation}
\begin{gathered}
\Psi_n(z):= K_n \abs{z}^{2-n}, \quad \textrm{ if } n>2\\
\Psi_2(z):= K_2 \log{\abs{z}}
\end{gathered}
\end{equation}
Following the terminology in \cite{GT}, we may occasionally refer to $\Delta^{-1}w$ as the Newtonian potential of $w$. 
Note that we now have an integro-differential equation and the usual Poisson equation. We will find a viscosity solution to the former equation and then use convolution with $\Psi_n$ to solve the Poisson equation.
Following \cite{AT}, we say that $u$ is a viscosity subsolution of 
$$F(x,\Delta^{-1}w, w,Dw,D^2w)=0$$
if and only if for all smooth test functions $\phi$ such that $u^*-\phi$ has a local maximum at $x$ and $u^*-\phi <0$ in $\mathbb{R}^n-\{x\}$ we have
$$F_*(x,\Delta^{-1}\phi,u^*,D\phi,D^2\phi)\leq 0.$$
Similarly, we may define a viscosity supersolution to the integro-differential problem.
Before we begin to show a Perron-type theorem for Equation (\ref{ZeroEqn}), we recall a basic lemma about $\Delta^{-1}$. 
\begin{lemma}
Suppose $\phi\in C_c^\infty(\mathbb{R}^n)$. Then $\Delta^{-1}\phi \in C^\infty(\mathbb{R}^n)$. In particular, we have the pointwise continuity of $\Delta^{-1}\phi$.
\end{lemma}
Now we proceed with the usual program towards a Perron-type existence theorem. 
\begin{theorem}
Let $\mathcal{F}$ be a nonempty family of viscosity subsolutions to 
$$F(x,\Delta^{-1}w, w,Dw,D^2w)=0$$
and let
$$v(x):=\sup\{w(x) : w\in \mathcal{F}\}.$$
Then, $v$ is a viscosity subsolution of the same problem.
\end{theorem}
\begin{proof}
Suppose $v^*-\phi$ has strict local maximum of zero at $x$, and, without loss of generality and for $r$ small enough, assume that it is the unique local maximum in $B_r(x)$. Now we can find a sequence of points $x_n$ in $B_r(x)$ and a sequence of functions $w_n$ in $\mathcal{F}$ such that 
$$v^*(x)=\lim_{n\to \infty} v(x_n) =\lim_{n\to \infty} w_n(x_n).$$
Let $y_n$ be a maximum point of $w_n^*-\phi$ on $B_r$. Then we have
$$v^*(y_n)-\phi(y_n)\geq w_n^*(y_n)-\phi(y_n)\geq w_n^*(x_n)-\phi(x_n)\geq w_n(x_n)-\phi(x_n).$$
The right-hand side converges to zero as $n\to \infty$, so every limit $y$ of the $y_n$ satisfies
$$v^*(y)-\phi(y)\geq 0$$
from which we conclude that $y=x$. Then, from the subsolution condition for $v_n$:
$$F(y_n,\Delta^{-1}\phi(y_n),v_n^*(y_n),D\phi(y_n),D^2(y_n))\leq0,$$
the previous lemma about the continuity of $\Delta^{-1}$, and the continuity of $F$, we arrive at the conclusion of the theorem.
\end{proof}
\begin{theorem}
Classical solutions of Equation (\ref{ZeroEqn}) are viscosity solutions.
\end{theorem}
\begin{proof}
Let $\phi$ be a test function such that $u-\phi$ has a local maximum at $x$. We may also suppose that $u(y)\leq \phi(y)$ for all $y\in \mathbb{R}^n-\{x\}$. Recalling our definition of $\Delta^{-1}$, we have
$$\Delta^{-1}(u-\phi)(x)= \int_{\mathbb{R}^n} (u-\phi)(y)\Psi(x-y)dy\leq 0$$
It follows that $\Delta^{-1}u\leq \Delta^{-1}\phi$. Moreover, since we have a local maximum at $x$, we have $Du(x)=D\phi(x)$ and $D^2u\leq D^2 \phi$. By our assumptions on $F$ we conclude that:
$$F(x,u(x),\Delta^{-1}\phi(x),D\phi(x),D^2\phi(x))\leq F(x,u(x),\Delta^{-1}u(x),Du(x),D^2u(x))\leq0$$
whence we conclude that $u$ is a viscosity solution to the same problem.

\end{proof}
\begin{theorem}[Perron Method]
Let $f$ and $g$ be respectively a viscosity subsolution and viscosity supersolution to our problem.
Suppose that $f_\ast>-\infty$ and $g^\ast<\infty$. If $f\leq g$ and $F$ satisfies the hypotheses placed on Equation (\ref{ZeroEqn}), then there exists a viscosity solution $u$ with $f\leq u \leq g$.
\end{theorem}
\begin{proof}
Let 
$$\mathcal{F}= \{\textrm{subsolutions } v : v\leq g\}.$$
This is a nonempty set by hypothesis. Thus, if we define 
$$u(x)= \sup \{v(x) : v\in \mathcal{F}\}$$
and use our previous theorem, we get another subsolution $u$ satisfying $u\leq g$. Now we show that $u$ is a supersolution of the same problem. Pick a test function $\phi$ such that $u_\ast -\phi$ has a relative minimum at $x$ and such that $u_\ast-\phi>0$ on $\mathbb{R}^n-\{x\}$. Locally, we may also assume that 
$$u_\ast(y)-\phi(y) \geq \abs{y-x}^4$$ on some small ball $B_r(x)$.
Assume by contradiction that $u$ is not a supersolution, and in fact that
$$F(x,u_\ast(x),\Delta^{-1}\phi(x),D\phi(x),D^2\phi(x)])>0.$$
Define a new function $w_\delta=\max\{\phi+\delta^4,u\}$. We show that for sufficiently small $\delta>0$, $w_\delta$ is a viscosity subsolution of our problem, the set $\{x : w(x)>u(x)\}$ is nonempty, and $w\leq g$. It follows that $w\in \mathcal{F}$ and is larger than $u$, a contradiction. \\
First, since $F$ is continuous, we can choose $\delta>0$ small enough so that
$$F(x,\phi(x)+\delta^4, \Delta^{-1}\phi,D\phi,D^2\phi)\leq 0$$ on $B_{2\delta}(x)$. It follows, by our monotonicity hypotheses on $F$, that $\phi+\delta^4$ is also a viscosity subsolution, whence $w$ is a subsolution by a previous theorem. That there exists a point $y$ where $w(y)>u(y)$ is clear since there must exist a sequence $x_n$ such that $u(x_n)\to u_\ast(x)$ and it is evident that
$$u(x_n)<\phi(x_n)+\delta^4\leq w(x_n).$$
We use our assumption in this contradiction proof to show that $w\leq g$. For this it suffices to show that for a small $\delta$ we have
$$\phi+\delta^4 \leq g \textrm{ on } B_\delta(x).$$
Then we can conclude that $u_\ast(x)=\phi(x)<g_\ast(x)$. Assume rather that $u_\ast(x)=g_\ast(x)$, so that $g_\ast-\phi$ has a local minimum at $x$, and by our assumption that $g_\ast$ is a viscosity subsolution that 
$$F(x,g_\ast(x),\Delta^{-1}\phi(x),D\phi(x),D^2(x))\geq 0$$
which contradicts our initial assumption.
\end{proof}
Lastly, we provide an existence theorem for our general equation.
\begin{theorem}
Assume the hypotheses of the previous theorem, and assume also that the given subsolution $f$ and the given supersolution $g$ are also in $L^p$ where $p>n/2$.
Then 
$$F(x,u,\Delta u, D(\Delta u), D^2(\Delta u))=0$$
has a solution in $W^{2,p}$.
\end{theorem}
\begin{proof}
As mentioned before, we consider our equation as a system of coupled equations:
$$F(x, \Delta^{-1}w, w, Dw, D^2w)=0$$
$$\Delta u = w$$
By our assumptions and the previous theorem, there is a viscosity solution $w$ to the first equation that is, moreover, in $L^p$ with $p>n/2$. We can then solve the second, Poisson equation, in the distributional sense for a function $u$ that is in $W^{2,p}$.
\end{proof}
\subsection{Nonlinear Eigenvalue Problem for the $\infty$-Bilaplacian}
In this subsection we always assume that $n\geq 3$.
Here we provide an example where we apply the existence theorem just proven. Consider the eigenvalue problem, where $0\leq f(x)\leq M$ is continuous.
\begin{equation}
\lambda u+(\Delta v)^3\abs{D(\Delta v)}^2=f(x)
\end{equation}
We express the equation as the coupled system:
\begin{equation}
\begin{gathered}
\lambda \Delta^{-1}w +w^3\abs{Dw}^2=f(x)\\
\Delta u=w
\end{gathered}
\end{equation}
The authors of \cite{JLM} apply the theory of viscosity solutions to a nonlinear eigenvalue problem associated to the second order $\infty$-Laplacian equation. We extend the range of the theory of viscosity solutions by dealing with an eigenvalue problem associated to the third order $\infty$-Bilaplacian equation. We prove:
\begin{theorem}
Let $n\geq 3$. For any $\lambda<0$ and for any continuous and bounded function $f$ with support on some open annulus centered at the origin, there is a viscosity solution in $C^{1,\alpha}(\mathbb{R}^n)$ for all $\alpha>0$ to:
$$\lambda u +(\Delta u)^3\abs{D(\Delta u)}^2=f(x).$$
\end{theorem}
\begin{proof}
Since $\lambda<0$, the monotonicity assumptions of our previous theorem are satisfied. It remains to find a subsolution and supersolution that are in $L^p$ for all $p>n/2$.We let $\mu, \nu, M>0$ be the positive constants such that $0\leq f(x)\leq M$ for all $x\in \mathbb{R}^n$ and $f\equiv 0$ away from the open annulus $\{ \nu< \abs{x} < \mu\}$.\\

We claim that $w=0$ is a subsolution. Indeed, if $(w-\phi)(y)$ has a local maximum of zero at some $x$ and $0 = w<\phi$ everywhere else, we have that:
$$0=w^3\abs{D\phi}^2\leq 0\leq f(x)+\Delta^{-1}\phi.$$

We claim that, for an appropriate constant $N$, the following function $w(x)$ is our desired supersolution:
$$w(x):= \begin{cases} 
0, & 0\leq \abs{x} \leq \nu\\
N \abs{x}^{1/2}, & \nu < \abs{x} < \mu \\
0, & \mu \leq \abs{x} \end{cases}
 $$
 First we notice by inspection that $w$ as constructed is lower-semi-continuous. Now suppose that $(w-\phi)(y)$ has a local minimum of zero at some $x$ and that $w(y)>\phi(y)$ for all $y\neq x$. Without loss of generality we may assume that $\phi$ is smooth and compactly supported in a ball of volume equal to $1$. We deal with three cases.
 
 In the first case, we have that $0\leq \abs{x} \leq \nu$, so that $w(x)=f(x)=0$. Then we just have to show that
 $$\lambda \Delta^{-1}\phi(x) \geq 0,$$
 which is evident since $\lambda<0$ and $\phi(y)\leq w(x)=0$ for all $y\in \mathbb{R}^n$. In the second case, when $\mu\leq \abs{x}$, we proceed in exactly the same manner. Thus, the only case remaining is when $\nu < \abs{x} < \mu$, which is what we assume for the rest of the proof.
 
 The function $w$ is continuously differentiable in an open neighborhood of $x$. Moreover, we see that
 $$\abs{Dw}^2(x) = 4N^2\abs{x}^{-1}$$
 so that
 $$w(x)^3\abs{Dw}^2(x) = 4N^5 \abs{x}^{1/2}.$$
 It follows that we must show 
$$4N^5 \abs{x}^{1/2} \geq f(x)-\lambda\Delta^{-1}\phi(x).$$
Since $f$ is bounded, it would suffice to prove $4N^5 \abs{x}^{1/2} \geq M-\lambda\Delta^{-1}\phi(x)$. We begin by appropriately bounding the Newtonian potential of $\phi$. 

For a dimensional constant $C_n$ we write:
$$\Delta^{-1}\phi(x) = C_n \int_{\abs{x-y}\leq 1} \frac{\phi(y)dy}{\abs{x-y}^{n-2}}+ C_n \int_{\abs{x-y}\geq 1} \frac{\phi(y)dy}{\abs{x-y}^{n-2}}.$$
 Thus, for another dimensional constant $C_n'$ we have
 $$\abs{\Delta^{-1}\phi(x)} \leq C_n \abs{\int \phi(y)dy} + C_n' \abs{\phi(x)}.$$
 Since $\phi(y)\leq w(y)$ for all $y$ and since the volume of the support of $\phi$ was assumed equal to $1$, we conclude that
 $$\abs{\Delta^{-1}\phi(x) }\leq C w(x)$$
 where we have abused notation slightly and let $C= C_n + C_n'$ be another dimensional constant. 

To finish the proof that $w$ is a supersolution, we now need only find a constant $N$ such that
$$4N^5 \abs{x}^{1/2} \geq M -\lambda C N \abs{x}^{1/2}$$
for all $x$ with $\nu <\abs{x} < \mu$. This, however, amounts to simply choosing $N$ so large that:
$$4N^5 +\lambda C N \geq \frac{M}{\sqrt{\nu}},$$
which we can certainly do. Our final observation is that $w$ as constructed is in $L^p(\mathbb{R}^n)$ for any $p$, since it is a bounded continuous function with compact support.

Finally, solving the Poisson equation $\Delta u= w$ and using the Sobolev embedding, we may conclude that the original eigenvalue problem has a positive viscosity solution that is in $C^{1,\alpha}(\mathbb{R}^n)$ for all $\alpha$.
\end{proof}

\end{document}